\newif\ifspringer
\springerfalse

\ifspringer
\RequirePackage{fix-cm}
\documentclass[smallcondensed, nospthms]{svjour3}     
\smartqed  
\else
\documentclass[reqno]{amsart}
\fi

\usepackage[utf8]{inputenc}
\usepackage{ae,aecompl}   
\usepackage[english]{babel}

\providecommand{\backrefalt}[4]{} 
\usepackage[setmiddle=colon]{mathdef}

\newtheorem*{belcher}{Belcher's Criterion}

\theoremstyle{definition}
\usepackage{bbding}
\newcommand{\sternsymbol}{\raisebox{-0.2em}{\Asterisk}}
\newtheorem*{stern}{\sternsymbol}

\begin{document}


\ifspringer


\title{On linear combinations of units with
  bounded coefficients and double-base digit expansions
  \thanks{Daniel Krenn and J\"org Thuswaldner are supported by the Austrian
    Science Fund (FWF): W1230, Doctoral Program ``Discrete Mathematics''.}
  \thanks{Daniel Krenn is supported by the Austrian Science Foundation (FWF):
    S9606, that is part of the Austrian National Research Network ``Analytic
    Combinatorics and Probabilistic Number Theory''.}
  \thanks{J\"org Thuswaldner and Volker Ziegler are supported by the ``Aktion
    \"Osterreich-Ungarn'', grant No. 80oeu6.}
}

\author{Daniel Krenn \and Jörg Thuswaldner \and Volker Ziegler}

\institute{D.\ Krenn \at
  Institute of Optimisation and Discrete Mathematics (Math B) \\
  Graz University of Technology \\
  Steyrergasse 30/II, A-8010 Graz, Austria \\
  \email{\href{mailto:math@danielkrenn.at}{math@danielkrenn.at} \textit{or}
    \href{mailto:krenn@math.tugraz.at}{krenn@math.tugraz.at}}
  \and
  J.\ Thuswaldner \at
  Chair of Mathematics and Statistics \\
  University of Leoben \\
  Franz-Josef-Strasse 18, A-8700 Leoben, Austria \\
  \email{\href{mailto:Joerg.Thuswaldner@unileoben.ac.at}{Joerg.Thuswaldner@unileoben.ac.at}}
  \and
  V.\ Ziegler \at
  Institute of Analysis and Computational Number Theory (Math A) \\
  Graz University of Technology \\
  Steyrergasse 30/II, A-8010 Graz, Austria \\
  \email{\href{mailto:ziegler@math.tugraz.at}{ziegler@math.tugraz.at}}
  \and}  

\date{Received: date / Accepted: date}


\else 


\title[On linear combinations of units]{On linear combinations of units with
  bounded coefficients and double-base digit expansions}



\date{\today}


\author[D. Krenn]{Daniel Krenn}

\address{\parbox{12cm}{%
    Daniel Krenn \\
    Institute of Optimisation and Discrete Mathematics (Math B) \\
    Graz University of Technology \\
    Steyrergasse 30/II, A-8010 Graz, Austria \\}}

\email{\href{mailto:math@danielkrenn.at}{math@danielkrenn.at} \textit{or}
  \href{mailto:krenn@math.tugraz.at}{krenn@math.tugraz.at}}


\author[J. Thuswaldner]{Jörg Thuswaldner}

\address{\parbox{12cm}{%
    Jörg Thuswaldner \\
    Chair of Mathematics and Statistics \\
    University of Leoben \\
    Franz-Josef-Strasse 18, A-8700 Leoben, Austria \\}}

\email{\href{mailto:Joerg.Thuswaldner@unileoben.ac.at}{Joerg.Thuswaldner@unileoben.ac.at}}


\author[V. Ziegler]{Volker Ziegler}

\address{\parbox{12cm}{%
    Volker Ziegler \\
    Institute of Analysis and Computational Number Theory (Math A) \\
    Graz University of Technology \\
    Steyrergasse 30/II, A-8010 Graz, Austria \\}}

\email{\href{mailto:ziegler@math.tugraz.at}{ziegler@math.tugraz.at}}


\thanks{Daniel Krenn and J\"org Thuswaldner are supported by the Austrian
  Science Fund (FWF): W1230, Doctoral Program ``Discrete Mathematics''.}

\thanks{Daniel Krenn is supported by the Austrian Science Foundation (FWF):
  S9606, that is part of the Austrian National Research Network ``Analytic
  Combinatorics and Probabilistic Number Theory''.}

\thanks{J\"org Thuswaldner and Volker Ziegler are supported by the ``Aktion
  \"Osterreich-Ungarn'', grant No. 80oeu6.}


\fi 


\ifspringer
\maketitle
\else
\fi

\begin{abstract} 
  
  Let $\ord$ be the maximal order of a number field. Belcher showed in the
  1970s that every algebraic integer in $\ord$ is the sum of pairwise distinct
  units, if the unit equation $u+v=2$ has a non-trivial solution
  $u,v\in\ord^*$. We generalize this result and give applications to
  signed double-base digit expansions.

\ifspringer
\fi
\end{abstract}

\ifspringer
\else
\maketitle
\fi


\section{Introduction}
\label{sec:intro}

In the 1960s Jacobson~\cite{Jacobson:1964} asked, whether the number fields
$\Q(\sqrt{2})$ and $\Q(\sqrt{5})$ are the only quadratic number fields such
that each algebraic integer is the sum of distinct
units. \'Sliwa~\cite{Sliwa:1974} solved this problem for quadratic number
fields and showed that even no pure cubic number field has this property. These
results were extended to cubic and quartic fields by
Belcher~\cite{Belcher:1974,Belcher:1976}. In particular, Belcher solved the
case of imaginary cubic number fields completely by applying the following
criterion, which now bears his name, cf.~\cite{Belcher:1976}.

\begin{belcher}\label{Belcher-th}
  Let $F$ be a number field and $\ord$ the maximal order of
  $F$. Assume that the unit equation
  \begin{equation*}
    u+v=2, \qquad u,v\in\ord^*
  \end{equation*}
  has a solution $(u,v)\neq(1,1)$.
  Then each algebraic integer in $\ord$ is the sum of distinct units.
\end{belcher}

The problem of characterizing all number fields in which every algebraic
integer is a sum of distinct units is still unsolved. Let us note that this
problem is contained in Narkiewicz' list of open problems in his famous
book~\cite[see page 539, Problem~18]{Narkiewicz:1974}.

Recently the interest in the representation of algebraic integers as sums of
units arose due to the contribution of Jarden and Narkiewicz
\cite{Jarden:2007}. They showed that in a given number field there does not
exist an integer $k$, such that every algebraic integer can be written as the
sum of at most~$k$ (not necessarily distinct) units. For an overview on this
topic we recommend the survey paper due to Barroero, Frei, and
Tichy~\cite{Barroero:2011}. Recently Thuswaldner and
Ziegler~\cite{Thuswaldner:2010} considered the following related problem. Let
an order $\ord$ of a number field and a positive integer $k$ be given. Does
each element $\alpha \in \ord$ admit a representation as a linear combination
$\alpha = c_1 \varepsilon_1 + \dots + c_\ell \varepsilon_\ell$ of units
$\varepsilon_1,\dots,\varepsilon_\ell \in \ord^*$ with coefficients $c_i\in
\{1,\dots, k\}$\,? This problem was attacked by using dynamical methods from
the theory of digit expansions. In the present paper we address this problem
again. In particular, we wish to generalize Belcher's criterion in a way to
make it applicable to this problem.

In order to get the most general form, we refine the definition of the unit sum
height given in \cite{Thuswaldner:2010}.

\begin{definition}
  Let $F$ be some field of characteristic $0$, $\Gamma$ be a finitely generated
  subgroup of $F^*$, and $R\subset F$ be some subring of $F$. Assume that
  $\alpha\in R$ can be written as a linear combination
  \begin{equation}\label{lincomb}
    \alpha=a_1\nu_1 +\dots+a_\ell \nu_\ell ,
  \end{equation}
  where $\nu_1,\dots,\nu_\ell \in \Gamma \cap R$ are pairwise distinct and
  $a_1\geq \dots \geq a_\ell>0$ are integers. If (in case there exists more
  than one representation of the form \eqref{lincomb}) $a_1$ in \eqref{lincomb}
  is chosen as small as possible, we call $\omega_{R,\Gamma}(\alpha)=a_1$ the
  \emph{$R$-$\Gamma$-unit sum height of $\alpha$}. In addition we define
  $\omega_{R,\Gamma}(0)\colonequals 0$ and
  $\omega_{R,\Gamma}(\alpha)\colonequals \infty$ if $\alpha$ admits no
  representation as a finite linear-combination of elements contained in
  $\Gamma \cap R$.  Moreover, we define
  \begin{equation*}
    \omega_\Gamma(R)=\max\set*{\omega_{R,\Gamma}(\alpha)}{\alpha\in R}
  \end{equation*}
  if the maximum exists. If the maximum does not exist we write
  \begin{equation*}
    \omega_\Gamma(R)=
    \begin{cases}
      \omega
      & \text{if $\omega_{R,\Gamma}(\alpha)< \infty$ for each $\alpha\in R$,}\\
      \infty
      & \text{if there exists $\alpha\in R$
        such that $\omega_{R,\Gamma}(\alpha)= \infty$.}\\
    \end{cases}
  \end{equation*}
\end{definition}

Let us note that for a number field $F$ with the group of units $\Gamma$ of an
order $\ord$ of $F$ we have $\omega_\Gamma(\ord) = \omega(\ord)$, where
$\omega(\ord)$ is the unit sum height defined in~\cite{Thuswaldner:2010}.

With those notations our main result is the following.

\begin{theorem}\label{Th:GenBelcher}
  Let $F\subset \C$ be a field and $\Gamma$ a finitely generated subgroup of
  $F^*$ with $-1\in\Gamma$. Let $R$ be a subring of $F$ that is
  generated as a $\Z$-module by a finite set $\cE \subset \Gamma \cap R$.
  Assume that for given integers $n\geq I \geq 2$ the equation
  \begin{equation}\label{Eq:Unit}
    u_1+\dots+u_I=n, \qquad u_1,\ldots,u_I \in \Gamma \cap R
  \end{equation}
  has a solution $(u_1,\ldots,u_I)\neq(1,\ldots,1)$. Then we have
  $\omega_\Gamma(R) \leq n-1$.
\end{theorem}

The following section, Section~\ref{sec:proof}, is devoted to the proof of
Theorem~\ref{Th:GenBelcher}. In the third section we apply our main theorem,
Theorem~\ref{Th:GenBelcher}, to some special orders of Shanks' simplest cubic
fields. A special case of that theorem yields applications to double-base
expansions. There we choose $F=\Q$, $R=\Z$ and $\Gamma=\langle -1,p,q\rangle$,
where $p$ and $q$ are coprime integers. We discuss that in
Section~\ref{sec:double-bases}.


\section{Proof of Theorem \ref{Th:GenBelcher}}
\label{sec:proof}


We start this section by giving a short plan of the proof.


\begin{proof}[Plan of Proof]
  Let $\alpha\in R$ be arbitrary. Our goal is to find a representation of
  $\alpha$ of the form \eqref{lincomb} in which the coefficients $a_1,\ldots,
  a_\ell$ are all bounded by $n-1$.  We first show that $\alpha$ can be
  represented as a linear combination of the form \eqref{lincomb} with
  $\nu_1,\ldots,\nu_\ell$ chosen in a particular way.  The idea of the proof is
  rather simple and is based on induction over the total weight of this
  representation (this is the sum of all of its coefficients, see
  Definition~\ref{Def:weight}). Start with a representation of $\alpha$ as
  above and choose a coefficient which is greater than or equal to $n$ (if such
  a coefficient does not exist, we are finished). Now
  apply~\eqref{Eq:Unit}. This leads to a new representation of $\alpha$ of the
  form \eqref{lincomb} whose total weight does not increase (and actually
  remains the same after excluding some trivial cases). This process is now
  repeated until we either have a representation in which all coefficients are
  bounded by $n-1$, or the support of the representation contains big gaps. In
  the first case we are finished. In the second case we can split the
  representation in two parts which are separated by a large gap. The total
  weight of each part is less than the total weight of the original
  representation of $\alpha$. We thus use the induction hypothesis on both of
  them, so we get a new representation of each part with coefficients bounded
  by $n-1$. Now, since the gap between the supports of these two parts is
  large, they do not overlap after we apply~\eqref{Eq:Unit} to them in the
  appropriate way and we can put them together to find a representation as
  desired also in this case.
\end{proof}


Now we start with the proof of Theorem~\ref{Th:GenBelcher}. First we introduce
some notations. For integers $a$ and $b$ we write
\begin{equation*}
  \intint{a}{b} \colonequals  \set{a,a+1,\dots,b}
\end{equation*}
for the integers in the interval from $a$ to $b$. For tuples
$\bfx=\tuple{x_1,\dots,x_M}$ and $\bfeps=\tuple{\eps_1,\dots,\eps_M}$ we set
\begin{equation*}
  \bfeps^\bfx \colonequals  \eps_1^{x_1} \dots \eps_M^{x_M}.
\end{equation*}


Observe first that each element of $R$ has at least one representation of the
form \eqref{lincomb}. The coefficients of that representation are integers, but
not necessarily smaller than $n$.


\begin{beh}\label{beh:existence}
  There exists a $K$\nbd-th root of unity $\zeta$, elements
  $\eta_1,\dots,\eta_L \in \cE$, and multiplicatively independent
  elements $\eps_1,\dots,\eps_M \in \Gamma \cap R$, abbreviated as
  $\bfeps=\tuple{\eps_1,\dots,\eps_M}$, such that 
  \begin{equation*}
    u_i = \zeta^{k_i} \bfeps^{\bfr^{(i)}}, \qquad i\in\intint{1}{I},
  \end{equation*}
  for some $k_1,\dots,k_I\in\intint{0}{K-1}$ and some
  $\bfr^{(1)},\dots,\bfr^{(I)} \in \Z^M$, each $\alpha \in R$ can be written
  as
  \begin{equation*}
    \alpha =
    \sum_{k\in\intint{0}{K-1}} \sum_{\ell\in\intint{1}{L}} \sum_{\bfx\in\Z^M}
    a_{k,\ell,\bfx} \zeta^k \eta_\ell \bfeps^\bfx
  \end{equation*}
  with non-negative integers $a_{k,\ell,\bfx}$, and such that no
  relation of the form
  \begin{equation*}
    \zeta^{k} \eta_i \bfeps^\bfx = \eta_j \bfeps^\bfy, \quad i \neq j
  \end{equation*}
  with integer exponents and $k\in\Z$ holds.
\end{beh}


\begin{proof}[Proof of \behref{beh:existence}]
  Let $u_1,\ldots, u_I$ be as in \eqref{Eq:Unit}.
  Choose a $K$\nbd-th root of unity $\zeta \in \Gamma \cap
  R$ (note that the torsion group of $\Gamma$ is finite and cyclic)
  and multiplicatively independent $\eps_1,\dots,\eps_M \in \Gamma \cap
  R$ with $M \leq I$, such that
  \begin{equation*}
    u_i = \zeta^{k_i} \eps_1^{r_1^{(i)}} \dots \eps_M^{r_M^{(i)}}
    = \zeta^{k_i} \bfeps^{\bfr^{(i)}} \qquad (i\in\intint{1}{I})
  \end{equation*}
  holds for some $\bfr^{(1)},\ldots,\bfr^{(I)} \in \Z^M$. We set
  \begin{equation}\label{errr}
    r\colonequals \max\set*{r_m^{(i)}}{i\in\intint{1}{I},m\in\intint{1}{M}}
  \end{equation}
  and want to mention that we reference to that $r$ later in this section.

  Let us consider a finite subset $\set{\eta_1,\dots,\eta_L} \subset
  \cE$ such that all $\alpha \in R$ can be written as a linear
  combination
  \begin{equation*}
    \alpha =
    \sum_{k\in\intint{0}{K-1}} \sum_{\ell\in\intint{1}{L}} \sum_{\bfx\in\Z^M}
    a_{k,\ell,\bfx} \zeta^k \eta_\ell \bfeps^\bfx
  \end{equation*}
  with $a_{k,\ell,\bfx}\in\Z$ (which is possible since $\cE$ finitely generates
  $R$ as $\Z$\nbd-module). We can (and do) choose that finite subset such
  that no relation of the form
  \begin{equation*}
    \zeta^k \eta_i \bfeps^\bfx = \eta_j \bfeps^\bfy, \quad i \neq j
  \end{equation*}
  with integer exponents and $k\in\Z$ holds.

  Note that $\zeta^k \eta_\ell \bfeps^\bfx \in \Gamma \cap R$.
  Furthermore, we can choose the coefficients $a_{k,\ell,\bfx}$ to be
  non-negative, since, by assumption, we have $-1\in\Gamma$, which
  allows us to choose the ``signs'' in our representation.
\end{proof}


From now on we suppose that $\zeta$, $\eta_1,\dots,\eta_L$, and
$\bfeps$ are fixed and given as in~\behref{beh:existence}. We use the
following convention on representations.


\begin{convention}
  Let $\alpha \in R$ and suppose we have a representation of $\alpha$
  where the coefficients are denoted by $a_{k,\ell,\bfx}$ (small
  Latin letter with some index), i.e., $\alpha$ is written as
  \begin{equation*}
    \alpha =
    \sum_{k\in\intint{0}{K-1}} \sum_{\ell\in\intint{1}{L}} \sum_{\bfx\in\Z^M}
    a_{k,\ell,\bfx} \zeta^k \eta_\ell \bfeps^\bfx
  \end{equation*}
  We denote by $A \subset \Z^M$ (capital Latin letter corresponding to
  the letter used for the coefficients) the minimal $M$-dimensional
  interval including all $\bfx$ with $a_{k,\ell,\bfx} \neq 0$. We write
  \begin{equation*}
    A = \intint{\underline{A}_1}{\overline{A}_1} \times \dots
    \times \intint{\underline{A}_M}{\overline{A}_M}.
  \end{equation*}
  We omit the range of the indices $k$ and $\ell$ since they are
  always the same. Thus $\alpha$ will be written as
  \begin{equation*}
    \alpha = \sum_{k,\ell}\sum_{\bfx\in A}
    a_{k,\ell,\bfx} \zeta^k \eta_\ell \bfeps^\bfx.
  \end{equation*}
\end{convention}


An important quantity is the weight of a representation. It is defined
as follows.


\begin{definition}\label{Def:weight}
  Let $\alpha\in R$ and suppose we have a representation as
  in~\behref{beh:existence}, i.e.,
  \begin{equation*}
    \alpha = \sum_{k,\ell}\sum_{\bfx\in A}
    a_{k,\ell,\bfx} \zeta^k \eta_\ell \bfeps^\bfx.
  \end{equation*}
  with non-negative integers $a_{k,\ell,\bfx}$. We call the minimum of all
  \begin{equation*}
    \sum_{k,\ell} \sum_{\bfx\in A} a_{k,\ell,\bfx}
  \end{equation*}
  among all possible representations (as above) of $\alpha$ the \emph{total
    weight of $\alpha$} and write $w_\alpha$ for it.
\end{definition}


As mentioned in the plan of the proof of Theorem~\ref{Th:GenBelcher}, we apply
Equation~(\ref{Eq:Unit}) to an existing representation to get another one. In
the following paragraph, we define that replacement step, which will then
always be denoted by \sternsymbol{}.


\begin{stern}[Replacement Step]\label{step:replace}
  Suppose we have a representation
  \begin{equation*}
    \alpha = \sum_{k,\ell}\sum_{\bfx\in A}
    a_{k,\ell,\bfx} \zeta^k \eta_\ell \bfeps^\bfx,
  \end{equation*}
  where at least one coefficient $a_{k,\ell,\bfx} \geq n$. We get a new
  representation by applying
  \begin{equation*}
    u_1 + \dots + u_I = n.
  \end{equation*}
  More precisely, if $u_i = \zeta^{k_i} \bfeps^{\bfr^{(i)}}$, then the
  coefficient $a_{k+k_i,\ell,\bfx+\bfr^{(i)}}$ is increased by $1$ for
  each $i \in \intint{1}{I}$ and $a_{k,\ell,\bfx}$ is replaced by
  $a_{k,\ell,\bfx}-n$.
\end{stern}


The following statements~\behref{beh:induction-base}
and~\behref{beh:few-summands} deal with two special cases.


\begin{beh}\label{beh:induction-base}
  If $\alpha\in R$ with $w_\alpha < I$, then
  Theorem~\ref{Th:GenBelcher} holds.
\end{beh}


We use that statement as the basis of our induction on the total
weight~$w$.


\begin{proof}[Proof of \behref{beh:induction-base}]
  Since $I \leq n$ we have $w_\alpha < n$. So the sum of all
  (non-negative) coefficients is smaller than $n$. Therefore all
  coefficients themselves are in $\intint{0}{n-1}$, which proves the
  theorem in that special case.
\end{proof}


From now on suppose we have an $\alpha \in R$ with a representation
\begin{equation*}
  \alpha = \sum_{k,\ell} \sum_{\bfx\in A}
  a_{k,\ell,\bfx} \zeta^k \eta_\ell \bfeps^\bfx,
\end{equation*}
which has minimal weight. That means, we have $w \colonequals w_\alpha$.


\begin{beh}\label{beh:few-summands}
  If $I<n$, then Theorem~\ref{Th:GenBelcher} holds.
\end{beh}


\begin{proof}[Proof of \behref{beh:few-summands}]
  Assume that there is a coefficient $a_{k,\ell,\bfx} \geq n$ in the
  representation of $\alpha$. We apply \sternsymbol{} to obtain a new
  representation. But because $I<n$, the new one has smaller total
  weight, which is a contradiction to the fact that $w$ was chosen minimal.
\end{proof}


Because of~\behref{beh:induction-base} and~\behref{beh:few-summands}
we suppose from now that $w \geq I$ and $I = n$.  As indicated
above, we prove Theorem~\ref{Th:GenBelcher} by induction on the total
weight $w$ of $\alpha$. More precisely we want to prove the following
claim by induction.


\begin{claim}\label{clm:main-claim}
  Assume that $\alpha \in R$ has a representation
  \begin{equation*}
    \alpha = \sum_{k,\ell} \sum_{\bfx\in A}
    a_{k,\ell,\bfx} \zeta^k \eta_\ell \bfeps^\bfx
  \end{equation*}
  with non-negative integers $a_{k,\ell,\bfx}$ and with minimal total
  weight $w$. Then $\alpha$ has also a representation of the form
  \begin{equation*}
    \alpha = \sum_{k,\ell} \sum_{\bfx\in G}
    g_{k,\ell,\bfx} \zeta^k \eta_\ell \bfeps^\bfx.
  \end{equation*}
  with integers $g_{k,\ell,\bfx} \in \intint{0}{n-1}$ and where
  \begin{equation*}
    G = \intint{\underline{A}_1-f(w)}{\overline{A}_1+f(w)}
    \times \dots \times \intint{\underline{A}_M-f(w)}{\overline{A}_M+f(w)}
  \end{equation*}
  with $f(1)=0$ and
  \begin{equation*}
    f(w) = T(w) r+f(w-1)\qquad (w\in \mathbb{N}),
  \end{equation*}
  where
  \begin{equation*}
    T(w) = \left(w+2(w-1)f(w-1)\right)^{Mw}
    K^w L^w.
  \end{equation*}
\end{claim}


In order to prove Theorem~\ref{Th:GenBelcher} it is sufficient to prove
Claim~\ref{clm:main-claim}. As already mentioned, we use induction on the total
weight~$w$ of $\alpha$. Note that the induction basis has been shown above
in~\behref{beh:induction-base}.


Let us start by looking what happens if one applies \sternsymbol{}.


\begin{beh}\label{beh:ess-new-repr}
  Repeatedly applying \sternsymbol{} yields pairwise ``essentially different''
  representations of $\alpha$.

  More precisely, by repeatedly applying \sternsymbol{}, it is not possible to
  get two representations
  \begin{equation*}
    \alpha
    = \sum_{k,\ell} \sum_{\bfx \in A} a_{k,\ell,\bfx} \zeta^k \eta_\ell \bfeps^\bfx
    = \sum_{k,\ell} \sum_{\bfx \in A} a_{k,\ell,\bfx} \zeta^k \eta_\ell \bfeps^{\bfx+\bfL}
  \end{equation*}
  with some $\bfL \in \Z^M \setminus \set{\bfzero}$.
\end{beh}


\begin{proof}[Proof of \behref{beh:ess-new-repr}]
  Remember that we assumed $I=n$. First, let us note that we have
  \begin{equation*}
    n \leq \sum_{i\in\intint{1}{n}} \abs{u_i}
  \end{equation*}
  because of Equation~\eqref{Eq:Unit}. Using the Cauchy-Schwarz inequality
  yields
  \begin{equation*}
    n^2\leq \left(\sum_{i\in\intint{1}{n}} 1\cdot \abs{u_i}\right)^2
    \leq n\sum_{i\in\intint{1}{n}} \abs{u_i}^2.
  \end{equation*}
  Hence,
  \begin{equation*}
    n<\sum_{i\in\intint{1}{n}} \abs{u_i}^2,
  \end{equation*}
  unless $\abs{u_1}=\dots=\abs{u_n}=1$ and $\sum_i u_i=n$, i.e.,
  $u_1=\dots=u_n=1$. Since the trivial solution has been
  excluded, we see that every application of \sternsymbol{} makes the
  quantity
  \begin{equation}\label{eq:quantity-larger}
    \sum_{k,\ell} \sum_{\bfx \in A} a_{k,\ell,\bfx}
    \left(\abs{\eps_1}^{x_1} \dots \abs{\eps_M}^{x_M}\right)^2
  \end{equation}
  larger, i.e., the quantity~\eqref{eq:quantity-larger} coming from
  coefficients $a_{k,\ell,\bfx}'$ is larger than~\eqref{eq:quantity-larger}
  from $a_{k,\ell,\bfx}$, where the $a_{k,\ell,\bfx}'$ are the
  coefficients after an application of \sternsymbol{} on a representation with
  coefficients  $a_{k,\ell,\bfx}$. Note that the $\eps_1,\dots,\eps_M$ are
  fixed, cf.\ statement~\behref{beh:existence}.

  Hence, repeatedly applying \sternsymbol{} produces pairwise disjoint
  representations. Moreover, we cannot get the same representation up to linear
  translation in the exponents twice, i.e., we cannot get representations
  \begin{equation*}
    \alpha
    = \sum_{k,\ell} \sum_{\bfx \in A} a_{k,\ell,\bfx} \zeta^k \eta_\ell \bfeps^\bfx
    = \sum_{k,\ell} \sum_{\bfx \in A} a_{k,\ell,\bfx} \zeta^k \eta_\ell \bfeps^{\bfx+\bfL}
  \end{equation*}
  with $\bfL \in \Z^M \setminus \set{\bfzero}$. Such a relation would imply
  that $\bfeps^\bfL=1$, which is a contradiction to the assumption
  that the $\eps_1,\dots,\eps_M$ are multiplicatively
  independent.
\end{proof}


Now we look what happens after sufficiently many applications of
\sternsymbol{}.


\begin{beh}\label{beh:other-repr-exists}
  Set
  \begin{equation*}
    T(w) \colonequals
    \left(
      w+2(w-1)f(w-1) \right)^{Mw}
    K^w L^w
  \end{equation*}
  and suppose we have a representation
  \begin{equation*}
    \alpha = \sum_{k,\ell} \sum_{\bfx\in A}
    a_{k,\ell,\bfx} \zeta^k \eta_\ell \bfeps^\bfx.
  \end{equation*}
  After at most $T(w)$ applications of \sternsymbol{} we get a
  representation
  \begin{equation*}
    \alpha = \sum_{k,\ell}\sum_{\bfx\in B}
    b_{k,\ell,\bfx} \zeta^k \eta_\ell \bfeps^\bfx,
  \end{equation*}
  such that one of the following assertions is true:
  \begin{enumerate}
  \item\label{enu:case-finished} Each coefficient satisfies $b_{k,\ell,\bfx} \in
    \intint{0}{n-1}$ and
    \begin{equation*}
      \overline{B}_m - \underline{B}_m
      \leq w + 2(w-1)f(w-1)
    \end{equation*}
    holds for all $m \in \intint{1}{M}$.
  \item\label{enu:case-gap} There exists an index $m$ such that
    \begin{equation*}
      \overline{B}_m - \underline{B}_m
      > w + 2(w-1)f(w-1)
    \end{equation*}
    holds.
  \end{enumerate}
\end{beh}


\begin{proof}[Proof of \behref{beh:other-repr-exists}]
  Each replacement step \sternsymbol{} yields an essentially different
  representation, see~\behref{beh:ess-new-repr}, and there are at most $T(w)$
  possibilities to distribute our new coefficients in an interval
  $\intint{0}{K-1}\times\intint{1}{L}\times B$ with
  \begin{equation*}
    \overline{B}_m - \underline{B}_m \leq w + 2(w-1)f(w-1)
  \end{equation*}
  for each $m$ with $1\leq m\leq M$. Therefore after at most $T(w)$ replacement
  steps we are either in case~\ref{enu:case-finished} or in
  case~\ref{enu:case-gap} of~\behref{beh:other-repr-exists}.
\end{proof}


\begin{beh}\label{beh:translation-small}
  With the setup and notations of~\behref{beh:other-repr-exists}, a possible
  ``translation of the indices'' stays small.

  More precisely, we have
  \begin{equation*}
    \max\set*{\abs{\underline{A}_m-\underline{B}_m}}{m\in\intint{1}{M}}
    \leq T(w) r,
  \end{equation*}
  and
  \begin{equation*}
    \max\set*{\abs{\overline{A}_m-\overline{B}_m}}{m\in\intint{1}{M}}
    \leq T(w) r,
  \end{equation*}
  where $r$ is as defined as in \eqref{errr}.
\end{beh}


\begin{proof}[Proof of \behref{beh:translation-small}]
  The quantity $r$ is the maximum of all exponents in the representation of the
  $u_i$ as powers of the $\eps_1,\dots,\eps_M$. Thus, an application of
  \sternsymbol{} can change the exponents, and therefore the upper and lower
  bounds, respectively, by at most $r$. We have at most $T(w)$ applications of
  \sternsymbol{}, so the statement follows.
\end{proof}


Now we look at the two different cases of
\behref{beh:other-repr-exists}. The first one leads to a result
directly, whereas in the second one we have to use the induction
hypothesis to get a representation as desired.


\begin{beh}\label{beh:case-finished}
  If we are in case~(\ref{enu:case-finished})
  of~\behref{beh:other-repr-exists}, then we are ``finished''.
\end{beh}


\begin{proof}[Proof of \behref{beh:case-finished}]
  Since
  \begin{equation*}
    \left|\overline{A}_m - \overline{B}_m\right|\leq T(w) r<T(w)r+f(w-1)=f(w)
  \end{equation*}
  and
  \begin{equation*}
    \left|\underline{A}_m - \underline{B}_m\right|\leq T(w) r<T(w)r+f(w-1)=f(w)
  \end{equation*}
  hold for each $m\in\mathbb{N}$ we have found a representation as
  desired in Claim~\ref{clm:main-claim}.
\end{proof}


\begin{beh}\label{beh:case-gap}
  If we are in case~(\ref{enu:case-gap})
  of~\behref{beh:other-repr-exists}, then we can split the
  representation into two parts and between them there is a ``large
  gap''.

  More precisely, there is a constant $c$ such that we can write $\alpha =
  \gamma + \delta$ with
  \begin{equation*}
    \gamma = \sum_{k,\ell} \sum_{\bfx \in B \atop x_m <c}
    b_{k,\ell,\bfx} \zeta^k\eta_\ell\bfeps^\bfx
  \end{equation*}
  and
  \begin{equation*}
    \delta = \sum_{k,\ell} \sum_{\bfx \in B \atop x_m >c + 2f(w-1)}
    b_{k,\ell,\bfx} \zeta^k\eta_\ell\bfeps^\bfx.
  \end{equation*}
 \end{beh}


\begin{proof}[Proof of \behref{beh:case-gap}]
  In case~\ref{enu:case-gap} of~\behref{beh:other-repr-exists} we have
  an index $m \in \intint{1}{M}$ with
  \begin{equation*}
    \overline{B}_m - \underline{B}_m
    \geq w + 2(w-1)f(w-1)
  \end{equation*}
  The total weight of $\alpha$ is $w$, so the representation
  \begin{equation*}
    \alpha = \sum_{k,\ell}\sum_{\bfx\in B}
    B_{k,\ell,\bfx} \zeta^k \eta_\ell \bfeps^\bfx,
  \end{equation*}
  has at most $w$ non-zero coefficients. Therefore, by the pigeon hole
  principle we can find an interval $J$ of length at least $2f(w-1)$
  and with the property that all coefficients $a_{\bfx, i}$ fulfilling
  $x_m \in J$ are zero. Therefore we can split up $\alpha$ as mentioned.
\end{proof}


\begin{beh}\label{beh:do-induction-step}
  If we have the splitting described in~\behref{beh:case-gap}, then
  Claim~\ref{clm:main-claim} follows for weight~$w$.
\end{beh}


\begin{proof}[Proof of \behref{beh:do-induction-step}]
  After renaming the intervals and coefficients, we have $\alpha =
  \gamma + \delta$ with
  \begin{equation*}
    \gamma = \sum_{k,\ell} \sum_{\bfx \in C}
    c_{k,\ell,\bfx} \zeta^k\eta_\ell\bfeps^\bfx
  \end{equation*}
  and
  \begin{equation*}
    \delta = \sum_{k,\ell} \sum_{\bfx \in D}
    d_{k,\ell,\bfx} \zeta^k\eta_\ell\bfeps^\bfx.
  \end{equation*}
  Both total weights~$w_\gamma$ and~$w_\delta$, respectively, are
  smaller than $w = w_\alpha$, so we can use induction hypothesis: We
  get representations
  \begin{equation}\label{gam}
    \gamma = \sum_{k,\ell} \sum_{\bfx \in E}
    e_{k,\ell,\bfx} \zeta^k\eta_\ell\bfeps^\bfx
  \end{equation}
  with $e_{k,\ell,\bfx} \in \intint{0}{n-1}$ and
  \begin{equation}\label{delt}
    \delta = \sum_{k,\ell} \sum_{\bfx \in F}
    f_{k,\ell,\bfx} \zeta^k\eta_\ell\bfeps^\bfx
  \end{equation}
  with $f_{k,\ell,\bfx} \in \intint{0}{n-1}$. The upper and lower bounds of the
  intervals in $C$ to $E$ differ by at most $f(w_\gamma) \leq f(w-1)$ in each
  coordinate. The same is valid for the intervals of~$D$ to $F$. Since the
  intervals in $C$ and $D$ were separated by intervals of length at least
  $2f(w-1)$, therefore the intervals in $E$ and $F$ are disjoint. In other
  words, the two representations in \eqref{gam} and \eqref{delt} do not
  overlap. So we can add these two representations and obtain
  \begin{equation*}
    \alpha = \sum_{k,\ell} \sum_{\bfx \in G}
    g_{k,\ell,\bfx} \zeta^k\eta_\ell\bfeps^\bfx
  \end{equation*}
  with $g_{k,\ell,\bfx} \in \intint{0}{n-1}$. We have
  \begin{equation*}
    \max \set*{\abs{\overline{G}_m-\overline{A}_m}}{m\in\intint{1}{M}}
    \leq T(w) r+f(w-1) = f(w)
  \end{equation*}
  and
  \begin{equation*}
    \max \set*{\abs{\underline{G}_m-\underline{A}_m}}{m\in\intint{1}{M}}
    \leq T(w) r+f(w-1) = f(w),
  \end{equation*}
  which finishes the proof.
\end{proof}

\section{The Case of Simplest Cubic Fields}
\label{sec:cubic-fields}

Let $a$ be an integer and let $\alpha$ be a root of the polynomial
\begin{equation*}
  X^3-(a-1)X^2-(a+2)X-1.
\end{equation*}
Then the family of real cubic fields $\Q(\alpha)$ is called the family of
Shanks' simplest cubic fields. These fields and the orders $\Z[\alpha]$ have
been investigated by several authors. In particular, in a recent paper of the
second and third author~\cite{Thuswaldner:2010} it was shown that the unit sum
height of the orders $\Z[\alpha]$ is $1$ in case of $a=0,1,2,3,4,6,13,55$ and
the unit sum height $\leq 2$ in case of $a=5$. Moreover, it was conjectured
that $\omega(\Z[\alpha])=1$ for all $a\in\Z$.

Using our main theorem we are able to prove the following result.

\begin{theorem}
  We have $\omega(\Z[\alpha])\leq 2$ for all $a\in\Z$.
\end{theorem}

\begin{proof}
  First let us note some important facts on $\Q(\alpha)$ and
  $\Z[\alpha]$, see for example Shanks' original
  paper~\cite{Shanks:1974}. We know that $\Q(\alpha)$ is Galois over
  $\Q$ with Galois group $G=\{\mathrm id, \sigma,\sigma^2\}$ and with
  $\alpha_2=\sigma(\alpha)=-1-\frac 1\alpha$. If we set
  $\alpha_1\colonequals \alpha$, then $\alpha_1$ and $\alpha_2$ are a
  fundamental system of units. Now we know enough about the structure
  of $\Z[\alpha]$ to apply Theorem~\ref{Th:GenBelcher}.

  If we can find three units $u_1,u_2,u_3\in\Z[\alpha]^*$ such that
  $u_1+u_2+u_3=3$ and $u_i \neq 1$, then the theorem is a direct consequence of
  Theorem~\ref{Th:GenBelcher}. Indeed we have
  \begin{equation}\label{Eq:3Darst}
    \begin{split}
      3 &=\overbrace{(\alpha_1^2+(-a+2)\alpha_1-a)}^{=u_1} \\
      &\phantom{=}+ \overbrace{(-2\alpha_1^2+(2a-1)\alpha_1 +a+4)}^{=u_2} \\
      &\phantom{=}+ \overbrace{(\alpha_1^2+(-a-1)\alpha_1-1)}^{=u_3} \\
      &= \alpha_1\alpha_2^2+\alpha_1^{-2}\alpha_2^{-1}+\alpha_1\alpha_2^{-1}. 
      \qedhere
    \end{split}
  \end{equation}
\end{proof}


\section{Application to Signed Double-Base Expansions}
\label{sec:double-bases}


We start with the definition of a signed double-base expansion of an integer.


\begin{definition}[Signed Double-Base Expansion]
  Let $p$ and $q$ be different integers. Let $n$ be an
  integer with
  \begin{equation*}
    n = \sum_{i\in\N_0, j\in\N_0} d_{ij} p^i q^j,
  \end{equation*}
  where $d_{ij} \in \set{-1,0,1}$ and only finitely many $d_{ij}$ are
  non-zero. Then such a sum is called a \emph{signed $p$-$q$-double-base
    expansion of $n$.} The pair $(p,q)$ is called \emph{base pair}.
\end{definition}


A natural first question is, whether each integer has a signed double-base
expansion for a fixed base pair.

If one of the bases $p$ and $q$ is either $2$ or $3$, then existence follows
since every integer has a \emph{binary representation} (base~$2$ with digit set
$\set{0,1}$) and a \emph{balanced ternary representation} (base~$3$ with digit
set $\set{-1,0,1}$), respectively. To get the existence results for general
base pairs, we use the following theorem, cf.~\cite{Birch:1959}


\begin{theorem}[Birch]
  Let $p$ and $q$ be coprime integers. Then there is a positive integer
  $\f{N}{p,q}$ such that every integer larger than $\f{N}{p,q}$ may be
  expressed as a sum of distinct numbers of the form $p^iq^j$ all with
  non-negative integers $i$ and $j$.
\end{theorem}


\begin{corollary}
  Let $p$ and $q$ be coprime integers. Then each integer has a signed
  $p$-$q$-double-base expansion.
\end{corollary}


Next we want to give an efficient algorithm that allows to calculate a signed
double base expansion of a given integer. Birch's theorem, or more precisely
the proof in~\cite{Birch:1959}, does not provide an efficient way to do that.
However, using our main result, there is a way to compute such expansions
efficiently at least for certain base pairs.


\begin{corollary}\label{cor:double-base-Z}
  Let $p$ and $q$ be coprime integers with absolute value at least $3$. If
  there are non-negative integers $x$ and $y$ such that
  \begin{equation}\label{eq:replace2}
    2 = \abs{p^x - q^y},
  \end{equation}
  then each integer has a signed $p$-$q$-double-base expansion which can be
  computed efficiently (there exists a polynomial time algorithm). In
  particular given a $p$-adic expansion of an integer $\alpha$, one has to apply
  \eqref{eq:replace2} at most $O(\log(\alpha)^2)$ times.
\end{corollary}


\begin{proof}
  We start to prove the first part of the corollary and therefore apply Theorem
  \ref{Th:GenBelcher} with $\F=\Q$, $R=\Z$ and $\Gamma$ is the multiplicative
  group generated by $-1,p$ and $q$.  Since by assumption $2=\pm(p^x-q^y)$ we
  have a solution to \eqref{Eq:Unit} and Theorem \ref{Th:GenBelcher} yields
  that $p$-$q$-double-base expansions exist.
    
  Now let us prove the statement on the existence of a polynomial time
  algorithm. Assume that for the integer $\alpha$ the $p$-adic expansion
  \begin{equation*}\label{eq:p-adic}
   \alpha=a_{0}+a_{1}p+\dots +a_{k} p^{k}
  \end{equation*}
  is given, with $a_{0},\ldots,a_{k}\in\intint{0}{p-1}$.  Let us note that the
  weight $w$ of this representation is at most $O(\log \alpha)$. Now the
  following claim yields the corollary.
\end{proof}
 
\begin{claim}
  Assume
  \begin{equation*}
    \alpha=\sum_{i\in\intint{0}{I}} a_i p^i
  \end{equation*}
  with $a_i\in \Z$ and $I\in\N_0$, and set $w = \sum_{i\in\intint{0}{I}}
  \abs{a_i}$. Then, after at most $\frac{w^2-w}2$ replacement steps
  \sternsymbol{} we arrive in a representation of the form
  \begin{equation*}
    \alpha=\sum_{j\in\intint{0}{J}} q^{jy} \sum_{k\in\intint{0}{K}} b_{k,j}p^k,
  \end{equation*}
  where the $b_{k,j}$ are integers with $\abs{b_{k,j}} \leq 1$, and $J,K\in\N_0$.
\end{claim}
  
\begin{proof}
  We prove the claim by induction on $w$. If $w\leq1$ the statement of the
  claim is obvious. Further, if all the $a_i$ are in $\set{-1,0,1}$ we are
  done. Therefore we assume that there is at least one index $i$ with
  $\abs{a_i}>1$.

  We now apply the replacement step \sternsymbol{} in the following way: If
  $a_i>1$, then $a_i$ is replaced by $a_i-2$, if $a_i<1$, then $a_i$ is
  replaced by $a_i+2$. After at most $w-1$ such steps, we get a new
  representation of the form
  \begin{equation*}
    \alpha = \sum_{i\in\intint{0}{I_c}} c_ip^i 
    + q^y \sum_{i\in\intint{0}{I_d}} d_ip^i,
  \end{equation*}
  $I_c,I_d\in\N_0$, $c_i,d_i\in\Z$, such that all $c_i$ fulfil
  $\abs{c_i}\leq1$. Note that no replacement step \sternsymbol{} increases the weight $w$.

  Now consider 
  \begin{equation*}
    \beta = \sum_{i\in\intint{0}{I_d}} d_ip^i.
  \end{equation*}
  The weight of $\beta$ fulfils
  \begin{equation*}
    w_\beta = \sum_{i\in\intint{0}{I_d}} \abs{d_i} \leq w-1,
  \end{equation*}
  since in each replacement step it is increased exactly by~$1$. Now, by
  induction, hypothesis we obtain a representation
  \begin{equation*}
    \beta = \sum_{j\in\intint{0}{J_e}} q^{jy} \sum_{k\in\intint{0}{K_e}} e_{k,j}p^k,
  \end{equation*}
  where the $e_{k,j}$ are integers with $\abs{e_{k,j}} \leq 1$ and
  $J_e,K_e\in\N_0$. Further, this can be done in $\frac{w_\beta^2-w_\beta}2$
  steps. Setting $b_{i,0} = c_i$ and $b_{i,k} = e_{i,k-1}$ for $k>0$ yields the
  desired representation. Moreover, this can be done
  with at most 
  \begin{equation*}
    \frac{w_\beta^2-w_\beta}2 + w-1
    \leq \frac{(w-1)(w-2)}2 + w-1
    = \frac{w(w-1)}2
  \end{equation*}
  applications of \sternsymbol{}, which finishes the proof of the claim.
\end{proof}


Now we want to give some examples for base pairs, where the corollary can be
used.


\begin{example}
  Let $(p,q)$ be a \emph{twin prime pair}, i.e., we have $q=p+2$ and both $p$
  and $q$ are primes. Then clearly
  \begin{equation*}
    2 = q - p,
  \end{equation*}
  so, by Corollary~\ref{cor:double-base-Z}, every integer has a signed
  $p$-$q$-double-base expansion, which can be calculated efficiently. 
\end{example}


\begin{example}
  Let $p=5$ and $q=23$. We have
  \begin{equation*}
    2 = 5^2 - 23,
  \end{equation*}
  therefore every integer has a signed $5$-$23$-double-base expansion, which
  can be calculated efficiently. Again Corollary~\ref{cor:double-base-Z} was
  used.

  To see some concrete expansions, we calculated the following:
  \begin{align*}
    995 &= - 5^{5} + 5^{4} + 5^{3} \cdot 23 -  5^{2} + 5 \cdot 23 + 23^{2} + 1 \\
    996 &= - 5^{3} + 5^{2} \cdot 23 + 23^{2} -  5 + 23 - 1 \\
    997 &= - 5^{3} + 5^{2} \cdot 23 + 23^{2} -  5 + 23 \\
    998 &= 5^{4} -  5^{3} -  5^{2} + 23^{2} -  5 - 1 \\
    999 &= 5^{4} -  5^{3} -  5^{2} + 23^{2} -  5 \\
    1000 &= 5^{4} -  5^{3} -  5^{2} + 23^{2} -  5 + 1 \\
    1001 &= - 5^{3} + 5^{2} \cdot 23 + 23^{2} + 23 - 1 \\
    1002 &= - 5^{3} + 5^{2} \cdot 23 + 23^{2} + 23 \\
    1003 &= 5^{4} -  5^{3} -  5^{2} + 23^{2} - 1
  \end{align*}
  In each case we started with an initial expansion, which is obtained
  by a greedy algorithm: For a $v\in\Z$ find the closest $5^i \cdot
  23^j$, change the coefficient for that base, and continue with $v -
  5^i \cdot 23^j$. Then we calculated the expansion by applying the
  equation $2 = 5^2 - 23$ as in the proof of
  Theorem~\ref{Th:GenBelcher}. The implementation\footnote{The source
    code can be found on \url{http://www.danielkrenn.at/belcher/}.
    Further a full list of expansions of the natural numbers up to
    $10000$ can be found there.}  was done in
  Sage~\cite{Stein-others:2012:sage-mathem-4.8}.
\end{example}


One can find pairs $(p,q)$ where Corollary~\ref{cor:double-base-Z} does not
work. The following remark discusses some of those pairs.


\begin{remark}
  Consider the equation
  \begin{equation}\label{eq:rem:thm-fails}
    2 = \abs{p^x - q^y}
  \end{equation}
  with non-negative integers $x$, $y$. A first example, where the corollary
  fails, is $p=5$ and $q=11$. Indeed, looking at Equation~\eqref{eq:rem:thm-fails}
  modulo~$5$ yields a contradiction. Another example is $p=7$ and $q=13$, where
  looking at \eqref{eq:rem:thm-fails} modulo~$7$, yields a contradiction. A
  third example is $p=7$ and $q=11$.
\end{remark}


So in the cases given in the remark above, as well as in a lot of other cases,
we cannot use the corollary to compute a signed double-base expansion
efficiently. This leads to the following question.


\begin{question}\label{con:exist-double-base}
  Is there an efficient (polynomial time) algorithm for each base pair $(p,q)$ to compute a
  signed $p$-$q$-double-base expansion for all integers?
\end{question}


There is also another way to use Theorem~\ref{Th:GenBelcher}. For some
combinations of $p$ and $q$ we can get a weaker result. First, we define an
extension of the signed double-base expansion: we allow negative exponents in
the $p^i q^j$, too.


\begin{definition}[Extended Signed Double-Base Expansion]
  Let $p$ and $q$ be different integers (usually coprime). Let $z\in\Q$. If we
  have
  \begin{equation*}
    z = \sum_{i\in\Z, j\in\Z} d_{ij} p^i q^j,
  \end{equation*}
  where $d_{ij} \in \set{-1,0,1}$ and only finitely many $d_{ij}$ are non-zero,
  then we call the sum an \emph{extended signed $p$-$q$-double-base expansion
    of $z$.}
\end{definition}


With that definition, we can prove the following corollary to
Theorem~\ref{Th:GenBelcher}.


\begin{corollary}\label{cor:double-base-Zp}
  Let $p$ and $q$ be coprime integers. If there are integers $a$, $b$, $c$, and
  $d$ with $(a,b,c,d)\neq(0,0,0,0)$ and such that
  \begin{equation}\label{eq:replace2-Zp}
    2 = p^a q^b \pm p^c q^d,
  \end{equation}
  then every element of $\Z[1/p,1/q]$ has
  an extended signed $p$-$q$-double-base expansion which can be computed
  efficiently (polynomial time algorithm).
\end{corollary}


\begin{remark}
If we have a solution to the equation in
Corollary~\ref{cor:double-base-Z}, then Corollary~\ref{cor:double-base-Zp}
works, too. But more can be said about the existence and efficient
computability of extended double-base expansions for the elements of
$\Z[1/p,1/q]$. If each integer has an efficient computable signed
$p$-$q$-double-base expansion, then each element of $\Z[1/p,1/q]$ has an
extended signed $p$-$q$-double-base expansion which can be computed
efficiently. This result is not difficult to prove.
\end{remark}

Now we prove the corollary.


\begin{proof}[Proof of Corollary~\ref{cor:double-base-Zp}]
  The proof of this corollary runs along the same lines as the proof
  of Corollary \ref{cor:double-base-Z}.
  
  We apply Theorem~\ref{Th:GenBelcher} with $\F=\Q$, $R=\Z[1/p,1/q]$
  and $\Gamma$ is the multiplicative group generated by $-1$, $p$ and
  $q$. Since, by assumption, $2=\pm(p^aq^b-p^cq^d)$ we have a solution
  to~\eqref{Eq:Unit}, Theorem~\ref{Th:GenBelcher} yields that
  $p$-$q$-double-base expansions exist.
  
  Next, we claim that we may assume $p$ and $q$ are odd and
  $p,q>3$. Indeed assuming that $p\in\set{2,3}$, then we can write
  $\alpha\in\Z[1/p,1/q]$ in the form
  \begin{equation*}
    \alpha =  \frac{\wt\alpha}{p^{x_p}q^{x_q}}
  \end{equation*}
  with $\wt\alpha\in\Z$ and appropriate exponents $x_p$ and
  $x_q$. Moreover, $\wt\alpha$ has a representation of the form
  \begin{equation*}
    \wt\alpha=\sum_{i\in\intint{0}{k}}a_i p^i
  \end{equation*}
  with $a_i\in\set{-1,0,1}$. However the computation of such a
  represenation can be done efficiently and takes polynomial time in
  the height $h(\alpha)$, where
  \begin{equation*}
    h(n/m)=\max \{\log \abs{n}, \log \abs{m}, 1\}
  \end{equation*}
  provided $n,m\in\Z$ are coprime.
  
  Since we may assume $p,q>3$, we want to show next that a solution to
  equation~\eqref{eq:replace2-Zp} necessarly takes the form
  \begin{equation*}
   2=\pm p^{-a}\pm p^{-a}q^b,
  \end{equation*}
  with $a,b\geq 0$. We observe that a solution to
  \eqref{eq:replace2-Zp} with $a,c>0$ or $b,c>0$ does not exist, since
  otherwise $p\divides2$ or $q\divides2$. Next we note that if $a\neq
  c$ ($b\neq d$ respectively) the $p$-adic valuation ($q$-adic
  valuation) on the right hand side of \eqref{eq:replace2-Zp} would be
  the minimum of $a$ and $c$ ($b$ and $d$ respectively) and in view of
  the left hand side, this minimum must be $0$.  Thus any solution to
  equation~\eqref{eq:replace2-Zp} must be of one of the following
  forms:
  \begin{align*}
   2 &= \pm p^aq^b \pm 1, \\
   2 &= \pm p^{-a}q^{-b} \pm p^{-a}q^{-b}, \\
   2 &= \pm p^{-a}\pm p^{-a}q^b, \\
   \intertext{or}
   2 &= \pm p^a\pm q^b,
  \end{align*}
  where $a$ and $b$ are positive integers. Obviously the first two
  cases have no solution and the last case has been treated in
  Corollary~\ref{cor:double-base-Z}.

  Now let us write $\alpha\in\Z[1/p,1/q]$ in the form
  \begin{equation*}
   \alpha=\frac{a_0+a_1p+\dots+a_kp^k}{q^{x_q}p^{x_p}}.
  \end{equation*}
  We are now in a similar situation as in the proof of
  Corollary~\ref{cor:double-base-Z}. Let $w=\sum_{i=1}^k
  \abs{a_i}$. Then by similar arguments as in
  Corollary~\ref{cor:double-base-Z} we find an extended signed
  $p$-$q$-double-base expansion of $\alpha$ with at most
  $\frac{w^2-w}2$ applications of \sternsymbol{}. Thus we have a
  polynomial in $h(\alpha)$ time algorithm.
\end{proof}


We can use the corollary proved above to get the following examples.


\begin{example}
  Let $p$ be a \emph{Sophie Germain prime} and $q=2p+1$. We obtain
  \begin{equation*}
    2 = q p^{-1} - p^{-1}.
  \end{equation*}
  Using Corollary~\ref{cor:double-base-Zp} yields that every element of
  $\Z[1/p,1/q]$ has an efficient computable extended signed $p$-$q$-double-base
  expansion. 

  The case when $p$ is a prime and $q=2p-1$ is a prime works analogously. 
  
\end{example}


The end of this section is dedicated to a short discussion. All the
results on efficient computability in this section needed a special
representation of $2$. We have given some pairs $(p,q)$ where the
methods given here do not work.

Further, one could ask, whether the representations we get have a
special structure. Of particular interest would be an algorithm to get
expansions with a small number of summands (small number of non-zero
digits). For a given base pair $(p,q)$ this leads to the following
question

\begin{question}
  How to compute a signed $p$-$q$-double-base expansion with minimal
  weight for a given integer?
\end{question}

A greedy approach for solving this question can be found in Berth\'e and
Imbert~\cite{Berthe-Imbert:2009}, some further results 
can be found in Dimitrov and
Howe~\cite{Dimitrov-Howe:2011:length-double-base-repr}.


\ifspringer
\bibliographystyle{spmpsci}      
\else
\bibliographystyle{abbrv}
\fi
\bibliography{Belcher}

\end{document}